\documentclass[11pt]{amsart}

\usepackage{amsmath, amssymb,  mathabx, amsfonts,enumerate}
\usepackage[all]{xy}
\usepackage{amscd,stmaryrd}
\usepackage{comment}
\usepackage{mathtools}
\usepackage{tikz} 
\usepackage{tikz-cd} 
\usepackage{lipsum}
\tikzcdset{diagrams={nodes={inner sep=7.5pt}}}

\usepackage{url}
\usepackage{hyperref}

\newtheorem{theorem}{Theorem}[section]
\newtheorem{lemma}[theorem]{Lemma}
\newtheorem{corollary}[theorem]{Corollary}
\newtheorem{proposition}[theorem]{Proposition}
\newtheorem{theoremletter}{Theorem}

 \theoremstyle{definition}
 \newtheorem{definition}[theorem]{Definition}

 \newtheorem{example}[theorem]{Example}

  \newtheorem*{example*}{Example}

\numberwithin{equation}{section}
 
\newcommand {\N}{\mathbb{N}} 
\newcommand {\Z}{\mathbb{Z}}

\newcommand{\LL}{\mathcal{L}}

\DeclareMathOperator{\Ker}{Ker}
\DeclareMathOperator{\End}{End}

\DeclareMathOperator{\im}{Im}

\DeclareMathOperator{\Id}{Id}

\begin{document}
\title[Linear NUCA: duality and dynamics]{On linear non-uniform cellular automata: duality and dynamics}    
\author[Xuan Kien Phung]{Xuan Kien Phung}
\address{Département d'informatique et de recherche opérationnelle,  Université de Montréal, Montréal, Québec, H3T 1J4, Canada.}
\email{phungxuankien1@gmail.com}   

\begin{abstract}
For linear non-uniform cellular automata (NUCA) over an arbitrary universe, we introduce and investigate their dual linear NUCA. Generalizing results for linear CA, we show that dynamical properties namely pre-injectivity, resp. injectivity, resp. stably injectivity, resp. invertibility of a linear NUCA is equivalent to  surjectivity, resp. post-surjectivity, resp. stably post-surjectivity, resp. invertibility of the dual linear NUCA. However,  while bijectivity is a dual property for linear CA, it  is no longer the case for linear NUCA.  
We prove that for linear NUCA, 
 stable injectivity and stable post-surjectivity are precisely characterized respectively by left invertibility and right invertibility and that a linear NUCA is invertible if and only if it is pre-injective and stably post-surjective. Moreover, we show  that linear NUCA satisfy the important shadowing property. 
 Applications on the dual surjunctivity are also obtained. 
\end{abstract}
\date{\today}
\maketitle
  
\setcounter{tocdepth}{1}

\section{Introduction} 
Let us recall some basic notions of symbolic dynamics. 
For a discrete set $A$ and a group $G$, a \emph{configuration} $c \in A^G$ is simply a map $c \colon G \to A$.  Two configurations $x,y  \in A^G$ are \emph{asymptotic} if  $x\vert_{G \setminus E}=y\vert_{G \setminus E}$ for some finite subset $E \subset G$.  
The \emph{Bernoulli shift} action $G \times A^G \to A^G$ is defined by $(g,x) \mapsto g x$, 
where $(gx)(h) =  x(g^{-1}h)$ for  $g,h \in G$ and $x \in A^G$. 
We equip the \emph{full shift} $A^G$ with the \emph{prodiscrete topology}. 
For $x \in A^G$, let $\Sigma(x) = \overline{\{gx \colon g \in G\}} \subset A^G$, i.e., the smallest subshift which contains $x$. 
Following the idea of von Neumann and Ulam  \cite{neumann}, a CA over a group $G$ (the \emph{universe}) and a set $A$ (the \emph{alphabet})  is  a self-map $A^G \righttoleftarrow$ which is $G$-equivariant and uniformly continuous (cf.~\cite{hedlund-csc}, \cite{hedlund}). 
When  different cells are allowed to admit different local transition maps, we have the more general notion of \emph{non-uniform CA} (NUCA) (cf. \cite[Definition~1.1]{phung-tcs}, \cite{Den-12a}, \cite{Den-12b}).   

\begin{definition}
\label{d:most-general-def-asyn-ca}
Let $G$ be a group and let  $A$ be a set. Let $M \subset G$ be a subset and let $S = A^{A^M}$ be the set of all maps $A^M \to A$. Given $s \in S^G$, the NUCA $\sigma_s \colon A^G \to A^G$ is defined for all $x \in A^G$ and $g \in G$ by the formula 
\begin{equation*}
\sigma_s(x)(g)=  
    s(g)((g^{-1}x)  
	\vert_M). 
 \end{equation*}
 \end{definition} 
\par 
The set $M$ is called a \emph{memory} and $s \in S^G$  the \textit{configuration of local defining maps} of $\sigma_s$. It follows that  every CA is a NUCA with finite memory and constant configuration of local defining maps. 
\par 
Following \cite{phung-tcs}, we say that $\sigma_s$ is  \emph{stably injective} if  $\sigma_p$ is injective  for every $p \in \Sigma(s)$ and that $\sigma_s$ is \emph{invertible} if it is bijective and the inverse map  $\sigma_s^{-1}$ is a NUCA with \emph{finite} memory. 
Moreover, $\sigma_s$ is \emph{left-invertible}, resp. \emph{right invertible}, if there exists a NUCA with \textit{finite} memory $\tau \colon A^G \to A^G$ such that $\tau \circ \sigma_s= \Id$, resp. $\sigma_s\circ \tau= \Id$.  The NUCA $\sigma_s$ is \emph{pre-injective} if $\sigma_s(x) = \sigma_s(y)$ implies $x= y$ whenever $x, y \in A^G$ are asymptotic, and $\sigma_s$ is \emph{post-surjective} if for all $x, y \in A^G$ with $y$ asymptotic to $\sigma_s(x)$, there exists $z \in A^G$ asymptotic to $x$ such that  $ \sigma_s(z)=y$. 
\par
In \cite{blanchard-maass-97}, the authors studied the  notion of post-surjectivity for CA and showed in particular that every pre-injective and post-surjective CA over a finite alphabet is invertible \cite[Theorem~1.1]{kari-post-surjective}. Consequently, they argued that post-surjectivity is a sort of “dual” to pre-injectivity: it is a strengthening of surjectivity, in a similar way that pre-injectivity is a weakening of injectivity. For linear CA, the duality is confirmed by the results obtained in \cite{bartholdi-duality-2017} where the author showed that every linear CA $\tau$ on a full shift admits a natural dual linear CA $\tau^*$. Moreover,  \cite[Theorem~1.4]{bartholdi-duality-2017} says that over a finite dimensional vector space alphabet,  a linear CA is injective, resp. pre-injective, if and only if its dual is post-surjective, resp. surjective.
\par 
The first goal of the paper is to investigate the notion of dual linear NUCA (see Section~\ref{s:dual-linear-nuca}) which extends the case of CA. By introducing the   dual configuration of local defining maps $s^*$ of every configuration $s$ of local defining maps of a linear NUCA $\sigma_s \colon V^G \to V^G$ with finite memory where $V$ is a vector space and $G$ is a group, we can define  the dual linear NUCA $\sigma_s^*$ as $\sigma_{s^*} \colon V^{*G} \to V^{*G}$   which is also a linear NUCA with finite memory. 
\par 
We then establish several pairs of dual properties for such NUCA. It turns out that stable injectivity and  invertibility are respectively dual properties of stable post-surjectivity and invertibility for linear  NUCA. Here, as for stable injectivity, we say that a  NUCA $\sigma_s$ is \emph{stably post-surjective} if $\sigma_p$ is post-surjective for every $p \in \Sigma(s)$ (see also Section~\ref{s:stable-post-surjectivity}). More specifically, we obtain the following generalization (Theorem~\ref{t:dual-properties-linear-dynamic}) to the case of linear NUCA of  the result \cite[Theorem~1.4]{bartholdi-duality-2017} for linear CA (see also Lemma~\ref{l:dual-property-1}):  
\par 
\begin{theoremletter}
\label{t:intro-dual-properties-linear-dynamic}
Let $M$ be a finite subset of a countable group $G$. Let $V$ be a finite dimensional vector space over a field and let $S = \LL(V^M, V)$.  Then for every  $s \in S^G$, the following hold: 
\begin{enumerate}[\rm (i)]
    \item $\sigma_s$ is pre-injective $\Longleftrightarrow$ $\sigma_{s^*}$ is surjective;  
    \item 
    $\sigma_s$ is injective  $\Longleftrightarrow$   $\sigma_{s^*}$ is post-surjective. 
    \item 
     $\sigma_s$ is stably injective  $\Longleftrightarrow$   $\sigma_{s^*}$ is stably post-surjective. 
\item 
$\sigma_s$ is invertible   $\Longleftrightarrow$   $\sigma_{s^*}$ is invertible. 
\end{enumerate}
\end{theoremletter}
\par 
By \cite[Theorem~1.4]{bartholdi-duality-2017} and \cite[Theorem~1.2]{kari-post-surjective}, a linear CA over a finite vector space alphabet is bijective if and only so is its dual linear CA. However, the above equivalence breaks down for linear NUCA (see Proposition~\ref{p:bijective-non-dual}). In other words, bijectivity is not a dual property for linear NUCA. 
\par 
In the case of linear NUCA, we obtain the following useful and compact characterization of stable injectivity and stable post-surjectivity respectively  in terms of left and right invertibility (see Theorem~\ref{t:characterization-left-right-invertibility}). 
\begin{theoremletter}
    \label{t:intro-characterization-left-right-invertibility}
Let $V$ be a finite  vector space and let $G$ be a countable group. Then a linear NUCA with finite memory over $V^G$ is stably injective, resp. stably post-surjective, if and only if it is left-invertible, resp. right invertible. 
\end{theoremletter}
\par 
Combining Theorem~\ref{t:intro-dual-properties-linear-dynamic}  and Theorem~\ref{t:intro-characterization-left-right-invertibility} with Theorem~\ref{t:invertible-general-linear}, we obtain another characterization of invertibility in terms of pre-injectivity and stable  post-surjectivity combined (see Theorem~\ref{t:characterization-invertible-linear-NUCA}). 
\par 
\begin{theoremletter}
    \label{t:intro-characterization-invertible-linear-NUCA}
Let $G$ be a countable group and let $V$ be a finite vector space. Suppose that $\tau \colon V^G \to V^G$ is a linear NUCA with finite memory. Then $\tau$ is invertible if and only if it is pre-injective and stably post-surjective if and only if it is surjective and stably injective. 
\end{theoremletter}
\par 
Since there exist injective non-stably injective linear NUCA with finite memory by \cite[Example~14.1]{phung-tcs}, we obtain post-surjective linear NUCA with finite memory which are not stably post-surjective as a consequence of  Theorem~\ref{t:intro-dual-properties-linear-dynamic}. Hence, we deduce from Theorem~\ref{t:intro-characterization-invertible-linear-NUCA} that 
\cite[Theorem~1.1]{kari-post-surjective} fails for the general class of linear NUCA with finite memory: to ensure the invertibility of a linear NUCA, surjectivity should be strengthened by stably post-surjecitivity, a strictly stronger property than post-surjectivity, to correctly compensate the weakening of stable injectivity by pre-injectivity.   
\par 
As another immediate  application of Theorem~\ref{t:intro-dual-properties-linear-dynamic} and the surjunctivity result \cite[Theorem~B]{phung-tcs} (see also \cite{gottschalk},  \cite{gromov-esav},  \cite{csc-sofic-linear}, \cite{cscp-alg-goe}, \cite{ceccherini}, \cite{phung-geometric}),  we obtain the following result concerning the dual surjunctivity (cf. \cite[Conjecture~2]{kari-post-surjective}, \cite{phung-post-surjective}) where under suitable assumptions, post-surjectivity or stable post-surjectivity alone implies the invertibility of linear NUCA with finite memory:   

\begin{theoremletter}
\label{t:intro-dual-surjunctivity-linear-nuca}
Let $M$ be a finite subset of a countable group $G$. Let $V$ be a finite  vector space and let $S = \LL(V^M, V)$. Let $s \in S^G$ be asymptotic to a constant configuration.  
Then $\sigma_s$ is invertible in each of the following cases: 
\begin{enumerate} [\rm (i)]
\item $G$ is amenable and $\sigma_s$ is post-surjective;
\item $G$ is residually finite and $\sigma_s$ is stably post-surjective. 
\end{enumerate}
\end{theoremletter}
\par 
The paper is organized as follows. In Section~\ref{s:induced-local-map}, we recall the useful construction of various induced local maps of NUCA. Then we establish in Section~\ref{s:CIP-linear-NUCA} the fundamental closed image property Theorem~\ref{t:closed-image-linear} for linear NUCA with finite memory and arbitrary finite dimensional vector space alphabets (not necessarily finite). 
In Section~\ref{s:stable-post-surjectivity}, we prove that for linear NUCA with finite memory, surjectivity and pre-injectivity are stable properties and that right-invertibility implies stable post-surjectivity (cf. Lemma~\ref{l:direct-stably-post-surjective-nuca-linear-finite-a}). 
Then in Section~\ref{s:dual-linear-nuca}, we define the notion of dual linear NUCA and describe the natural pairing which extends the pairing construction in \cite{bartholdi-duality-2017}. In particular, we prove in Lemma~\ref{l:duality-functoriality} the functorial property under composition of  dual NUCA, i.e., $(\sigma \circ \tau)^*=\tau^* \circ  \sigma^*$ for all linear NUCA $\sigma, \tau \colon V^G \to V^G$. The main technical Lemma \ref{l:dual-property-1} given in 
Section~\ref{s:main-lemma}, together with Lemma~\ref{l:duality-functoriality}, will be essential for the proof of Theorem~\ref{t:intro-dual-properties-linear-dynamic} in Section~\ref{s:dual-preoperties-main-result}. 
In Section~\ref{s:weak-uniform-post-surj}, we give a weak uniform post-surjectivity property for post-surjective linear NUCA with finite memory over an arbitrary finite dimensional vector space alphabet.  
The proofs  of Theorem~\ref{t:intro-characterization-invertible-linear-NUCA},  Theorem~\ref{t:intro-characterization-left-right-invertibility}, and Theorem~\ref{t:intro-dual-surjunctivity-linear-nuca} are presented in 
Section~\ref{s:stably-post0surjective},  
Section~\ref{s:characterization-left-inevertible}, and  
Section~\ref{s:dual-surjunctivity} respectively. 
We give in 
Section~\ref{s:bijectivity-not-dual} a concrete example of a one-dimensional linear NUCA with finite memory and finite alphabet which is bijective but its dual is not injective. 
Finally, Section~\ref{s:kurka-construction} and 
Section~\ref{s:main-result} are dedicated to show that, as for linear CA, the pseudo-orbit tracing property or also known as the shadowing property is also verified by linear NUCA with finite memory over finite dimensional vector space (see Theorem~\ref{t:main-shadow}).

\section{Induced local maps of NUCA} 
\label{s:induced-local-map}

Let $G$ be a group and let $A$ be a set. For every subset $E\subset G$, $g \in G$,  and $x \in A^E$ we define $gx \in A^{gE}$ by setting $gx(gh)=x(h)$ for all $h \in E$. In particular,  $gA^E= \{gx \colon x \in A^E\}=A^{gE}$. 
\par 
Let $M$  be a subset of a group $G$. Let $A$ be a set and let $S=A^{A^M}$. For every finite subset $E \subset G$  and $w \in S^{E}$,  
we define a map  $f_{E,w}^+ \colon A^{E M} \to A^{E}$ for every $x \in A^{EM}$ and $g \in E$ by (in the case of CA, see e.g. \cite[Lemma~3.2]{cscp-alg-goe}, \cite[Proposition~3.5]{phung-2020}, \cite[Secion~2.2]{phung-embedding}):  
\begin{align}
\label{e:induced-local-maps} 
    f_{E,w}^+(x)(g) & = w(g)((g^{-1}x)\vert_M). 
\end{align}
\par 
In the above formula, note that  $g^{-1}x \in A^{g^{-1}EM}$ and $M \subset g^{-1}EM$ since $1_G \in g^{-1}E$ for $g \in E$. Therefore, the map   $f_{E,w}^+ \colon A^{E M} \to A^{E}$ is well defined. 
\par
Consequently, for every $s \in S^G$, we have a well-defined induced local map $f_{E, s\vert_E}^+ \colon A^{E M} \to A^{E}$ for every finite subset $E \subset G$ which satisfies: 
\begin{equation}
\label{e:induced-local-maps-general} 
    \sigma_s(x)(g) =  f_{E, s\vert_E}^+(x\vert_{EM})(g)
\end{equation}
for every $x \in A^G$ and $g \in E$. Equivalently, we have for all $x \in A^G$ that: 
\begin{equation}
\label{e:induced-local-maps-proof} 
    \sigma_s(x)\vert_E =  f_{E, s\vert_E}^+(x\vert_{EM}). 
\end{equation}
\par 
It follows that for $\Gamma = \im \sigma_s$, we have the relation: 
\begin{equation} 
\label{e:local-image}
    \Gamma_E = \im f_{E, s\vert_E}^+. 
\end{equation}

\section{The closed image property of linear NUCA} 
\label{s:CIP-linear-NUCA}

\begin{theorem}
\label{t:closed-image-linear} 
Let $M$ be a finite subset of a countable group $G$. Let $V$ be a finite dimensional vector space and let $S=\LL(V^M, V)$. Then for every $s \in S^G$, the image $\sigma_s(V^G)$ is closed in $V^G$ with respect to the prodiscrete topology. 
\end{theorem}

\begin{proof}
Since $G$ is countable, we can find an exhaustion of $G$ by finite subsets $(E_n)_{n \in \N}$ so that $G = \cup_{n \in \N} E_n$. Up to enlarging $M$, we can clearly suppose without loss of generality that $1_G \in M$. 
Let $\Gamma = \sigma_s(V^G)$ and let $d \in V^G$ belong to the closure of $\Gamma$. Then for every $n \in \N$, the affine subspace 
\[ 
W_n \coloneqq (f^+_{E_n, s\vert_{E_n}})^{-1}(d\vert_{E_n}) \subset V^{E_nM}
\] 
is nonempty. Since $E_nM \subset E_mM$ for $m \geq n \geq 0$, we have canonical projection linear maps $\pi_{n,m} \colon V^{E_mM} \to V^{E_nM}$. It is clear that $\pi_n(W_m) \subset W_n$ for every $m \geq n \geq 0$ thus we obtain a well-defined  inverse system 
$(W_n)_{n \in \N}$ 
with affine transition maps $\pi_{n,m}\vert_{W_m} \colon W_m \to W_n$ where $m \geq n \geq 0$. 
Since the affine spaces $W_n$ are nonempty, we infer from \cite[Lemma~3.1]{cscp-jpaa} that $\varprojlim_{n  } W_n \neq \varnothing$. Thus, we can choose $x \in \varprojlim_{n  } W_n \subset \varprojlim_{n  } V^{E_nM} = V^G$. In particular, $x \vert_{E_nM} \in W_n$ for every $n \in \N$ and it follows from the definition of $W_n$ that we have:
\[
\sigma_s(x)\vert_{E_n} = f^+_{E_n, s\vert_{E_n}}(x\vert_{E_nM})= d\vert_{E_n}. 
\]
\par 
As $\cup_{n \in \N}E_n = G$, we deduce that $\sigma_s(x) = d$ and thus $ d \in \Gamma$.  We conclude that $\Gamma$ is closed in $V^G$ with respect to the discrete topology. 
\end{proof}

\section{On some  stable properties of NUCA} 
\label{s:stable-post-surjectivity}
By \cite[Theorem~A]{phung-tcs}, we know that reversibility is a stable property for NUCA over finite alphabets while injectivity is not. In other words, reversibility is equivalent to stable reversibility but stable injectivity is strictly stronger than injectivity. We show below that surjectivity is in fact equivalent to stable  surjectivity. 
\par 
As for stable injectivity, we define  stable surjectivity and stable pre-inejctivity for NUCA. 

\begin{definition}
Let $M$  be a subset of a group $G$ and let $S = A^{A^M}$ where $A$ is a set. 
Given $s \in S^G$, the NUCA $\sigma_s$ is said to be  \emph{stably surjective}, resp. stably pre-injective, if for every $p \in \Sigma(s)$, the NUCA $\sigma_p$ is surjective, resp. pre-injective. 
\end{definition}

\par 
\begin{lemma}
\label{l:stable-surjective-linear}
Let $M$ be a finite subset of a countable group $G$. Let $V$ be a finite dimensional vector space over a field and let  $s \in S^G$ where $S = \LL(V^M, V)$. Then $\sigma_s$ is surjective, resp. pre-injective,  if and only if it is stably surjective, resp. stably pre-injective.  
\end{lemma}

\begin{proof}
Up to enlarging $M$, we can suppose without loss of generality that $M=M^{-1}$ and $1_G \in M$.  
Since $G$ is countable, we can find an exhaustion of $G$ by finite subsets $(E_n)_{n \in \N}$ so that $G= \cup_{n \in \N} E_n$. 
Let $p \in \Sigma(s)$. Then there exists a sequence $(g_n)_{n \in \N}$ of elements in $G$  such that $(g_ns)\vert_{E_n} = p\vert_{E_n}$. 
Suppose first that $\sigma_s$ is surjective, then so is $\sigma_{g_ns}$ for every $n \in \N$ by \cite[Lemma~5.1]{phung-tcs}. Let $\Gamma = \im \sigma_p$ then we deduce from \eqref{e:local-image} 
that 
\begin{align*}
    \Gamma_{E_n} = \im  f_{E, p\vert_{E_n}}^+ =  \im  f_{E, (g_ns)\vert_{E_n}}^+ = (\im \sigma_{g_ns})\vert_{E_n}=V^{E_n}.
\end{align*}
 \par 
 Consequently, we find that $\Gamma$ is dense in $V^G$ with respect to the prodiscrete topology. As $\Gamma$ is also closed in $V^G$ by Lemma~\ref{t:closed-image-linear}, we deduce that $\Gamma= V^G$. Therefore, $\sigma_p$ is surjective for every $p \in \Sigma(s)$ and we conclude that $\sigma_s$ is stably surjective. 
 Since the converse is trivial, we conclude that $\sigma_s$ is surjective if and only if it is stably surjective. 
 \par 
Suppose on the contrary that $\sigma_s$ is pre-injective but $\sigma_p$ is not pre-injective. Then there exists $x \in V^G\setminus \{0^G\}$ asymptotic to $0^G$ such that $\sigma_p(x)=0^G$. Let us choose integers $m \geq n \geq 0$ such that $x \vert_{G\setminus E_n} =0^{G\setminus E_n}$ and $E_n M^2 \subset E_m$. Then since 
$p \vert_{E_m} = (g_ms)\vert_{E_m}$ and since $M$ is a memory set of both $\sigma_{g_ms}, \sigma_{p}$, we deduce from \eqref{e:induced-local-maps-general} that 
$\sigma_{g_ms}(x)(g)= \sigma_{p}(x)(g)=0$ for all $g \in E_nM$. On the other hand, for all $g \in G \setminus E_nM$, we have $\sigma_{g_ms}(x)(g)=0$ since $x\vert_{G \setminus E_n}=0^{G \setminus E_n}$.  It follows that $\sigma_{g_m s}(x)=0^G$ and thus  $\sigma_{g_ms}$ is not pre-injective. Consequently, $\sigma_s$ is not pre-injective by \cite[Lemma~5.1]{phung-tcs}, which is a contradiction. Hence, if $\sigma_s$ is pre-injective then it is stably pre-injective. The converse is obvious since $s \in \Sigma(s)$. Therefore, the proof of the lemma is complete.     
\end{proof}

\par 
By a similar proof, observe that surjectivity and pre-injectivity are also stable properties for NUCA with finite memory over an alphabet. 
\begin{lemma}
    \label{l:stable-pre-injective-surjective-nuca-}
Let $M$ be a finite subset of a countable group $G$. Let $A$ be a finite alphabet and let  $s \in S^G$ where $S = A^{A^M}$. Then $\sigma_s$ is surjective, resp. pre-injective, if and only if it is stably surjective, resp. stably pre-injective.  \qed 
    
\end{lemma}

Similarly, we obtain the equivalence between surjectivity and stable surjectivity for NUCA over finite alphabets.  

\begin{lemma}
Let $M$ be a finite subset of a countable group $G$. Let $A$ be a finite alphabet and $s \in S^G$ where $S = A^{A^M}$. Then $\sigma_s$ is surjective if and only if it is stably surjective. 
\end{lemma}

\begin{proof}
    The proof is the same, \emph{mutatis mutandis}, as the proof of Lemma~\ref{l:stable-surjective-linear} . The only needed modification is that we apply the closed image property for NUCA with finite memory over finite alphabets \cite[Theorem~4.4]{phung-tcs} instead of Lemma~\ref{t:closed-image-linear}. 
\end{proof}
\par 
The next result is a generalization of  \cite[Lemma~7.6]{phung-weakly} to the class of linear NUCA which says that a right-invertible linear NUCA is post-surjective. 

\begin{lemma}
\label{l:direct-stably-post-sur-nuca-linear}
 Let $M$ be a finite subset of a group $G$. Let $V$ be a   vector space   and let  $s,t \in S^G$ where $S = \LL(V^M, V)$. Suppose that  $\sigma_s \circ \sigma_t= \Id$. Then $\sigma_s$ is 
post-surjective.  
\end{lemma}

\begin{proof}
Let $x, y \in V^G$ be two configurations such that $y$ and $\sigma_s(x)$ are asymptotic. Then $w=y-\sigma_s(x)$ is asymptotic to $0^G$ and thus 
so is $z=\sigma_t(w)$. Since $\sigma_s \circ \sigma_t= \Id$, we deduce that: 
\begin{align*}
\sigma_s(x+z) & = \sigma_s(x) + \sigma_s(z) 
\\& = 
\sigma_s(x) + \sigma_s(\sigma_t(y-\sigma_s(x))) 
 \\
 & = \sigma_s(x) + \sigma_s(\sigma_t(y)) - \sigma_s(\sigma_t(\sigma_s(x))) \\
 &= \sigma_s(x) + y -  \sigma_s(x) \\
 & = y. 
\end{align*}
\par 
As $x+z$ is asymptotic to $x$, we can thus conclude that $\sigma_s$ is stably post-surjective. 
\end{proof}

\par 
When the alphabet is a finite vector space, the above result can be strengthened as follows. 
\begin{lemma}
\label{l:direct-stably-post-surjective-nuca-linear-finite-a}
 Let $M$ be a finite subset of a countable group $G$. Let $V$ be a finite vector space   and let $s,t \in S^G$ where $S = \LL(V^M, V)$. Suppose that $\sigma_s \circ \sigma_t= \Id$. Then $\sigma_s$ is stably 
post-surjective.  
\end{lemma}

\begin{proof}
Let $p \in \Sigma(s)$. Then since 
 $\sigma_s \circ \sigma_t= \Id$ and since $V$ is finite, 
 there exists by \cite[Theorem~11.1]{phung-tcs} a configuration $q \in \Sigma(t)$ such that $\sigma_p \circ \sigma_q= \Id$. Consequently, 
we infer from Lemma~\ref{l:direct-stably-post-sur-nuca-linear} that $\sigma_{p}$ is post-surjective. As the configuration $p \in \Sigma(s)$ can be arbitrary, we conclude that $\sigma_s$ is stably post-surjective and the proof is complete.  
\end{proof}

\section{Dual of linear NUCA}
\label{s:dual-linear-nuca}
\subsection{Dual configurations of local defining maps}

Fix a group $G$ and a vector space $V$ over a field. 
Let $M \subset G$ be a  finite subset and let $s \in S^G$ where $S =  \LL(V^M, V)$. Then by linearity, we define $s(g,m) \in \End(V)$ for every $m \in M$ and $g \in G$ by the formula: 
\begin{equation} 
s(g)(v) = \sum_{m \in M} s(g,m)v(m), \quad \text{ for all } v \in V^M.     
\end{equation} 
\par 
By extending $s$ by the zero map outside of $M$, i.e., by setting $s(g,m)=0$ for $m \in G \setminus M$, we obtain for every $v \in V^G$ that: 
\begin{equation}
    s(g)(v) = \sum_{h \in G} s(g,h)v(h). 
\end{equation}

\par 
\begin{definition}
\label{d:dual-config-local-map}
Let $M$ be a  finite subset of a group  $G$. Let  $V$ be a vector space over a field and let $S = \LL(V^M, V)$. Let $T = \LL(V^{*M^{-1}}, V^*)$ where $V^*$ is the dual space of $V$. We use the right superscript $^\mathsf{T}$ to denote the transpose of linear maps. Let $s \in S^G$.  The \emph{dual configuration of local defining maps} $s^* \in T^G$  is defined by setting for all $g,m \in G$: 
\begin{equation} 
\label{e:def-dual-s}
s^*(g,m)\coloneqq  
    s(gm,m^{-1})^\mathsf{T}. 
\end{equation} 
\par 
Moreover, we say that $\sigma_s^*\coloneqq  \sigma_{s^*}$ is the \emph{dual} linear NUCA of $\sigma_s$. 
\end{definition}
\par 

Consequently, with the above notations, we find for every $v \in V^{*G}$ that: 
\begin{equation}
\label{e:def-dual-s-remark}
    s^*(g)(v) = \sum_{m \in G} s^*(g,m)v(gm) = \sum_{m \in G} s(gm,m^{-1})^\mathsf{T}v(gm). 
\end{equation}
\par
Note also that $s^* \in T^G$ since \eqref{e:def-dual-s} implies that $s^*(g,m) \neq 0$ only if $m^{-1} \in M$ only if $m \in M^{-1}$. 
\par 

\begin{lemma}
\label{l:dual-dual}
Let $M$ be a  finite subset of a group  $G$. Let  $V$ be a vector space over a field and let $S = \LL(V^M, V)$. Then for every $s \in S^G$, we have $s^{**}=s$ and in particular, 
$\sigma_{s}^{**}=\sigma_s$. 
\end{lemma}

\begin{proof}
Let $g,m \in G$ then we infer directly from the relation  \eqref{e:def-dual-s} that: 
\begin{align*}
    s^{**}(g,m) = s^*(gm, m^{-1})^\mathsf{T}= s(g, m). 
\end{align*}
\par 
Consequently, we deduce that 
$s^{**}=s$ and 
it follows from the definition that we have:
\begin{equation*}
    \sigma_{s}^{**}=\sigma_{s^*}^*= \sigma_{s^{**}}=\sigma_s.
\end{equation*}
\par 
The proof is thus complete. 
\end{proof}

\par 

\begin{lemma}
\label{l:duality-family} 
    Let $M$ be a finite  subset of a countable group $G$. Let  $V$ be a vector space over a field and let $S = \LL(V^M, V)$. Then for every $s \in S^G$, we have mutually inverse canonical $G$-equivariant bijections $\varphi \colon \Sigma(s) \to \Sigma(s^*)$, $ p \mapsto p^*$ and 
    $\phi \colon \Sigma(s^*) \to \Sigma(s)$, $q \mapsto q^*$. 
\end{lemma}

\begin{proof}
Let $(E_n)_{n \in \N}$ be an exhaustion of $G$ by finite groups. Then $(E_nM)_{n \in \N}$ is also an exhaustion of $G$ by finite groups. 
Let $p \in \Sigma(s)$ and $q = p^*$. Then there exists a sequence $(g_n)_{n \in \N}$ of elements in $G$ such that we have $p\vert_{E_nM} = (g_ns)\vert_{E_nM}$ for all $n \in \N$. 
\par 
We claim that $q \in \Sigma(s^*)$. 
Indeed, for $m \in M$ and $g \in E_n$ with $n \in \N$, we find that $s^*(g_n^{-1}g,m) = s(g_n^{-1}gm, m^{-1})^\mathsf{T}$ and:  
\begin{align*} 
q(g,m) & = p(gm,m^{-1})^\mathsf{T}= (g_ns)(gm, m^{-1})^\mathsf{T}
\\ &=  s(g_n^{-1}gm, m^{-1})^\mathsf{T} = s^*(g_n^{-1}g,m)\\
& = (g_ns^*)(g,m).
\end{align*} 
\par 
Consequently, we deduce that $q\vert_{E_n M} = (g_ns^*)\vert_{E_nM}$ for every $n \in \N$ and thus 
$q \in \Sigma(s^*)$ as claimed. Similarly, since $s^{**}=s$ by Lemma~\ref{l:dual-dual}, we find that $q^* \in \Sigma(s)$ for all $q \in \Sigma(s^*)$. Again by Lemma~\ref{l:dual-dual}, 
we have $p^{**}=p$ and $q^{**}=q$ for all $p \in \Sigma(s)$ and $q \in \Sigma(s^*)$. In other words, $\phi \circ \varphi= \Id$ and $\varphi \circ \phi = \Id$. 
\par 
Finally, we can check that $\varphi$ and $\phi$ are $G$-equivariant as follows. For all $p \in \Sigma(s)$, $q=p^*$, $g, h\in G$, and $m \in M$, we find that: 
\begin{align*} 
     (hp)^*(g, m) = (hp)(gm, m^{-1})^\mathsf{T} = p(h^{-1}gm, m^{-1})^\mathsf{T}=(hp^*)(g,m).  
\end{align*} 
and the conclusion of the lemma follows. 
\end{proof}

\begin{lemma}
    \label{l:duality-functoriality} 
    Let $M$ be a finite  subset of a countable group $G$. Let  $V$ be a vector space over a field and let $S = \LL(V^M, V)$. Then for every $s, t \in S^G$, we have $(\sigma_s \circ \sigma_t)^*= \sigma_{t^*}\circ \sigma_{s^*}$. 
\end{lemma}

\begin{proof}
Up to enlarging $M$, we can suppose without loss of generality that $M=M^{-1}$. 
By the proof of \cite[Lemma~6.2]{phung-tcs}, we know that $M^2$ is a memory set of $\sigma_s \circ \sigma_t$. 
It follows that $(M^2)^{-1}=(M^{-1})^2=M^2$ is a memory set of $(\sigma_s \circ \sigma_t)^*$ and of $ \sigma_{t^*}\circ \sigma_{s^*}$. 
Moreover, for $p \in Q^G$ where  $p(g) \in A^{A^{M^2}}$  is  the map given for every $g \in G$  by: 
\[ 
p(g) =   s(g) \circ f_{M, g^{-1}t\vert_{M}}^+ , 
\]  
then we have $\sigma_s \circ \sigma_t= \sigma_p$ and thus $(\sigma_s \circ \sigma_t)^*= \sigma_p^*= \sigma_{p^*}$.  
Similarly, 
for $q \in Q^G$ where  $q(g) \in A^{A^{M^2}}$ is defined for every $g \in G$ by  
\[ 
q(g) = t^*(g) \circ f_{M, g^{-1}s^*\vert_{M}}^+, 
\] 
then we have $\sigma_{t^*} \circ \sigma_{s^*}= \sigma_q$. We claim that $p^*=q$. 
Indeed, for all $g, m \in G$,  
we have:  
 \begin{align*}
     p^*(g,m) & = p(gm, m^{-1})^\mathsf{T}\\ & = \sum_{k \in G} \left(s(gm, k) ((gm)^{-1}t)(k, k^{-1}m^{-1}) \right)^\mathsf{T}
     \\ & = \sum_{k \in G} \left(s(gm, k) t(gmk, k^{-1}m^{-1}) \right)^\mathsf{T}
     \\ & = \sum_{h \in G}  t(gh, h^{-1})^\mathsf{T} s(gm, m^{-1}h)^\mathsf{T} & (h=mk).   
 \end{align*}
 \par 
 On the other hand, we find that: 
 \begin{align*}
     q(g,m) & = \sum_{h \in G} t^*(g,h) (g^{-1}s^*) (h, h^{-1}m) \\
      & = \sum_{h \in G} t(gh, h^{-1})^\mathsf{T}s^* (gm, m^{-1}h)\\
      & = \sum_{h \in G} t(gh, h^{-1})^\mathsf{T} s(gm, m^{-1}h)^\mathsf{T}
      \\ & = p^*(g,m).
 \end{align*}
 \par 
 Therefore, since $g,m \in G$ are arbitrary, we deduce that $p=q^*$ as claimed. Consequently, $(\sigma_s \circ \sigma_t)^*= \sigma_{t^*}\circ \sigma_{s^*}$ and the proof is complete. 
\end{proof}

\subsection{The natural pairing} 

Let $V$ be a finite dimensional vector space over a field $k$ and let $G$ be a group.  Let $V^*G \coloneqq \bigoplus_{g \in G} V^*$ be the subset of configurations of $V^{*G}$ asymptotic to $0^G$. 
Then we have a natural pairing (cf.~\cite{bartholdi-duality-2017}) 
defined for all $(\omega, c ) \in V^*G \times V^G$ by the formula: 
\begin{align*}
\langle \cdot  \vert \cdot \rangle \colon V^*G \times V^G & \longrightarrow k \\ 
(\omega  , c) & \longmapsto  \langle \omega \vert c \rangle \coloneqq \sum_{g \in G} \langle \omega(g) \vert c(g) \rangle
\end{align*}
where we write $\langle \omega(g) \vert c(g) \rangle = \omega(g)(c(g))$. It is immediate that the pairing is nondegenerate since it is pointwise and $V$ is finite dimensional.  

\begin{lemma}
\label{l:duality-pairing}
Let $M$ be a finite  subset of a group  $G$. Let  $V$ be a vector space over a field and let $S = \LL(V^M, V)$. Then we have: 
\begin{equation} 
\langle \sigma_s^* (\omega) \vert c \rangle  = \langle   \omega \vert  \sigma_s(c) \rangle \quad \text{ for all }  s \in S^G, \omega \in V^*G, \text{ and } c \in V^G. 
\end{equation} 
\end{lemma} 

\begin{proof}
For every $s \in S^G$, $\omega \in V^*G$, and $c \in V^G$, we infer from the linearity and the definition of $\sigma_{s}^*$  the following equalities of \emph{finite} sums:
\begin{align*}
\langle \sigma_s^* (\omega) \vert c \rangle  & = \sum_{g \in G} 
\langle \sigma_{s^*}  (\omega)(g)  \vert c(g) \rangle \\
& = \sum_{g \in G} 
\langle {s^*}(g)(g^{-1}\omega)  \vert c(g) \rangle \\
& = 
\sum_{g \in G} \sum_{m \in G} \langle s(gm,m^{-1})^\mathsf{T} \omega(gm)
\vert c(g) \rangle 
\\
& = \sum_{g \in G} \sum_{h \in G} \langle s(h, h^{-1}g)^\mathsf{T} \omega(h)
\vert c(g) \rangle & (h\coloneqq gm)
\\
& =  \sum_{g \in G} \sum_{h \in G} \langle \omega(h)
\vert s(h, h^{-1}g) c(g) \rangle 
\\
& = 
 \sum_{h \in G} \sum_{k \in G} \langle \omega(h)
\vert s(h, k) c(hk)  \rangle & (k\coloneqq h^{-1}g) 
\\
& =  \sum_{h \in G}   \langle \omega(h) 
\vert \sigma_s(c)(h)  \rangle 
\\ 
& = \langle \omega \vert \sigma_s(c) \rangle. 
\end{align*} 
 \par
We conclude that $\langle \sigma_s^* (\omega) \vert c \rangle = \langle \omega \vert \sigma_s(c) \rangle$ and the proof is complete. 
\end{proof} 

\section{Main lemma} 
\label{s:main-lemma}
The following elementary lemma will be useful in the sequel. 

\begin{lemma}
\label{l:closedness-finite-total-linear}
Let $V$ be a finite dimensional vector space and let $G$ be a countable group. Suppose that $\Gamma \subsetneq  A^G$ is a closed linear subspace with respect to the prodiscrete topology. Then there exists a finite subset $\Omega \subset G$ such that $\Gamma_\Omega \subsetneq A^\Omega$.  
\end{lemma}

\begin{proof}
See the proof of \cite[Lemma~8.2]{phung-tcs} which is in fact valid for an arbitrary alphabet.  
\end{proof}
\par 
We shall now state the following key lemma of the duality of linear NUCA. The  proof is a direct generalization of \cite[Theorem~1.4]{bartholdi-duality-2017} and we include it here for the sake of completeness.  

\begin{lemma}
\label{l:dual-property-1}
Let $M$ be a finite  subset of a countable group $G$. Let  $V$ be a finite-dimensional vector space over a field and let $S = \LL(V^M, V)$. Then for every $s \in S^G$, we have: 
\begin{enumerate} [\rm (i)]
    \item $\Ker (\sigma_s\vert VG)^\perp  = \im \left(\sigma_{s^*}\vert (V^*)^G\right)$, 
    \item $\Ker \left(\sigma_s\vert V^G\right)^\perp = \im (\sigma_{s^*} \vert V^*G)$,
    \item $\im (\sigma_s\vert VG)^\perp = \Ker \left(\sigma_{s^*}\vert (V^*)^G\right)$, 
    \item $\im \left(\sigma_s\vert V^G\right)^\perp = \Ker (\sigma_{s^*}\vert V^*G)$. 
\end{enumerate}
\end{lemma}

\par 

\begin{proof}[Proof of Lemma~\ref{l:dual-property-1}] 
Up to enlarging $M$, we can suppose without loss of generality  that $1_G \in M$ and that $M$ is symmetric, i.e. $M=M^{-1}$. In particular, both $\sigma_s$ and $\sigma_{s^*}$ admit $M$ as a memory set. 
\par 
For (i), let  $y= \sigma_{s^*}(x)$ for some $x \in (V^*)^G$. Let $z \in \Ker (\sigma_s\vert VG)$ so that $z \in VG$ and $\sigma_s(z)=0^G$. We deduce from Lemma~\ref{l:dual-dual} and  Lemma~\ref{l:duality-pairing} that: 
\begin{align*}
    \langle z \vert y \rangle 
    = \langle z \vert \sigma_{s^*}(x) \rangle 
     = \langle \sigma_s(z)  \vert x \rangle =  \langle 0^G \vert x \rangle =0.
\end{align*}
\par 
Therefore, $\im \left(\sigma_{s^*}\vert (V^*)^G\right) \perp \Ker (\sigma_s\vert VG)$ and it follows that 
\[ 
\im \left(\sigma_{s^*}\vert (V^*)^G\right) \subset \Ker (\sigma_s\vert VG)^\perp.
\]
\par 
It is straightforward that the same argument gives the inclusion $"\supset"$ in the relations (ii), (iii), and (iv). 
\par 
For the other inclusion of assertion (i), let us  suppose on the contrary that 
$\Ker (\sigma_s\vert VG)^\perp \not \subset \Gamma$ where 
$\Gamma = \im \left(\sigma_{s^*}\vert (V^*)^G\right)$. In particular, $\Gamma \subsetneq (V^*)^G$ and   there exists $c \in \Ker (\sigma_s\vert VG)^\perp \setminus \Gamma$. 
Since $\Gamma$ is closed in $(V^*)^G$ by Theorem~\ref{t:closed-image-linear}, we deduce from Lemma~\ref{l:closedness-finite-total-linear} the existence of a finite subset $E \subset G$  such that $\Gamma_E\subsetneq (V^*)^E$ is a proper vector subspace and $c \vert_E \notin \Gamma_E$. Since  $(V^*)^E$ is finite dimensional, we can easily find $\omega \in \left((V^*)^E\right)^*= V^E$ which vanishes (under the pairing $\langle \cdot \vert \cdot \rangle$) on $\Gamma_{E}$ but not on $c\vert_E$. 
\par 
We regard $\omega$ as an element of $VG$ by simply extending $\omega$ by zeros outside of $E$. 
Since the natural pairing $\langle \cdot \vert \cdot \rangle$ is pointwise, it follows that $\omega \perp \Gamma$. 
Consequently, we infer from Lemma~\ref{l:duality-pairing} that $\sigma_s(\omega) \perp (V^*)^G$ since for every $x \in (V^*)^G$, we have: 
\begin{equation*}
\langle \sigma_s(\omega) \vert x \rangle = \langle \omega \vert \sigma_{s^*}(x) \rangle = 0. 
\end{equation*}
\par 
Therefore, $\sigma_s(\omega)=0^G$ as the pairing $\langle \cdot \vert \cdot \rangle$ is nondegenerate. We deduce that $ \omega \in \Ker (\sigma_s\vert VG)$. However,   since $\langle \omega \vert c\rangle \neq 0$, we obtain a contradiction to the assumption $c \perp \Ker (\sigma_s\vert VG)$. We conclude that $\Ker(\sigma_s\vert VG)^\perp  \subset \Gamma$ and therefore $\Ker(\sigma_s\vert VG)^\perp  = \Gamma$. Hence, point (i) is proved. 
\par 
For the inclusion "$\subset$" in the relations (iii) and (iv), let $c\in (V^*)^G$ such that $c \perp \im(\sigma_s\vert VG)$. Then a direct application of Lemma~\ref{l:dual-dual} and  Lemma~\ref{l:duality-pairing} shows  that $\sigma_{s^*}(c) \perp VG$. As the pairing is nondegenerate, we obtain $\sigma_{s^*}(c)=0^G$ and thus $c \in \Ker \left(\sigma_{s^*}\vert (V^*)^G\right)$. 
\par 
Therefore, 
$\im (\sigma_s\vert VG)^\perp \subset \Ker \left(\sigma_{s^*}\vert (V^*)^G\right)$. Similarly, we obtain by the same argument that 
$\im \left(\sigma_s\vert V^G\right)^\perp \subset \Ker (\sigma_{s^*}\vert V^*G)$.  Thus, assertions (iii) and (iv) are proved. 
\par 
For the remaining inclusion $\Ker \left(\sigma_s\vert V^G\right)^\perp \subset \im (\sigma_{s^*} \vert V^*G)$,  let $\omega \in V^*G$ be such that  $\omega \notin   \im (\sigma_{s^*} \vert V^*G)$. Note that $\omega$ has finite support. Since $ \im (\sigma_{s^*} \vert V^*G)$ is a vector subspace of $   V^*G$ which does not contain  $\omega$, we can clearly choose a linear form $z \in (V^*G)^*=V^G$  which vanishes on $ \im (\sigma_{s^*} \vert V^*G)$ but does not vanish on $\omega$ (with respect to the pairing $\langle \cdot \vert \cdot \rangle$).  
\par 
Hence, Lemma~\ref{l:dual-dual} and  Lemma~\ref{l:duality-pairing} imply that $\sigma_s(z) \perp V^* G$. Consequently, $\sigma_s (z) = 0^G$ because the pairing $\langle \cdot \vert \cdot \rangle$ is nondegenerate. It follows that $z \in\Ker \left(\sigma_s\vert V^G\right) $ and thus by the choice of $z$, we have $\omega \notin \Ker \left(\sigma_s\vert V^G\right)^\perp$. 
\par 
We conclude that  $\Ker \left(\sigma_s\vert V^G\right)^\perp \subset \im (\sigma_{s^*} \vert V^*G)$ and 
the proof of the lemma is  complete.
\end{proof}

\section{Weak uniform post-surjectivity}
\label{s:weak-uniform-post-surj}

We have the following weak uniform post-surjectivity result for the class of linear NUCA with finite memory whose alphabet is a finite dimensional vector space over a field (not necessarily finite). Our proof follows closely the strategy of the  proof of \cite[Lemma~5.3]{phung-post-surjective}. 

\begin{lemma}[Weak uniform post-surjectivity] 
\label{l:uniform-post-surjectivity-linear}
Let $M$ be a finite subset of a countable group $G$. Let $V$ be a finite dimensional vector space over a field. Let $S= \LL(V^M, V)$ and $s \in S^G$. Suppose that $\sigma_s$ is post-surjective. Then there exists a finite subset $E\subset G$ with the following property. For all $x, y\in V^G$ such that $y\vert_{G \setminus \{1_G\}} =\sigma_s(x)\vert_{G \setminus \{1_G\}}$, there exists $x' \in V^G$ such that $\sigma_s(x')=y$ and 
$x'\vert_{G \setminus E}= x\vert_{G \setminus E}$. 
\end{lemma}

\begin{proof} 
Up to enlarging $M$, we can suppose that $1_G \in M$ and $M=M^{-1}$. 
Let $(E_n)_{n \in \N}$ be an exhaustion of $G$ consisting of finite subsets such that $1_G \in E_0$. For every $n \in \N$, let us define a linear subspace $V_n$ of $V^G$: 
\begin{equation}
\label{e:uniform-post-surjectivity-eq-1}
V_n = 
\{ x\in V^G \colon \sigma_s(x) \vert_{G \setminus \{1_G\}} = 0^{G \setminus \{1_G\}}, x\vert_{G \setminus E_n} = 0^{G \setminus E_n} \}. 
\end{equation}
\par 
Consider the following linear map
\begin{align*}
    \varphi_n \colon V^{E_n}\times \{0\}^{E_n M^2 \setminus E_n} \to V^{E_nM}
\end{align*}
given by the formula $\varphi_n(x)(g) = s(g)((g^{-1}x)\vert_M)$ for all $x \in V^{E_n}\times \{0\}^{E_n M^2 \setminus E_n}$ and $g \in E_nM$. 
Note that $1_G \in E_n M$ for every $n\in \N$. We denote respectively by $p_n \colon  V^{E_nM} \to V^{\{1_G\}}$ and $q_n \colon V^{E_nM} \to V^{E_nM\setminus \{1_G\}}$ the canonical projections. 
Since $x\vert_{G\setminus E_n}=\{0\}^{G\setminus E_n}$ for all $n \in \N$, we can  identify 
$V_n=\Ker q_n \circ \varphi_n$ 
 as a linear subspace of $V^{E_n}\times \{0\}^{E_n M^2 \setminus E_n}=V^{E_n}$. 
For every $n \in \N$, consider the following linear subspace of $V$: 
\begin{equation*}
    Z_n = p_n (\varphi_n (V_n))= \sigma_s(V_n)_{\{1_G\}}. 
\end{equation*} 
\par 
As $E_n \subset E_{n+1}$,  \eqref{e:uniform-post-surjectivity-eq-1} implies that $V_n \subset V_{n+1}$. Consequently,  
$Z_n \subset Z_{n+1}$ for all $n \in \N$. 
We claim that $\cup_{n \in \N}Z_n=V$. 
Indeed, let $y \in V$ and consider $c \in V^G$ where $c(g)=0$ for all $g \in G\setminus \{1_G\}$ and $c(1_G)=y$.  Since $\sigma_s$ is post-surjective and since $\sigma_s(0^G)=0^G$,  there exist $x \in V^G$ and $n \in \N$ such that $x\vert_{G \setminus E_n}=0^{G \setminus E_n}$ and $\sigma_s(x)=c$. We deduce that $\sigma_s(x)\vert_{G\setminus \{1_G\}}=0^{G\setminus \{1_G\}}$ and thus $x \in V_n$. Moreover, as $\sigma_s(x)(1_G)=y$, it follows that $y \in Z_n$. Hence, the claim 
 $\cup_{n \in \N}Z_n=V$ is proved. 
\par 
Since $V$ is a finite dimensional vector space, we deduce that $Z_N=V$ for some $N \in \N$.  We shall see in the sequel that $E = E_N$ satisfies the desired property in the conclusion of the lemma. 
Indeed, let $x , y \in V^G$ such that $y\vert_{G \setminus \{1_G\}} =\sigma_s(x)\vert_{G \setminus \{1_G\}}$. Let  $c=y-\sigma_s(x)  \in V^G$ thus $c\vert_{G\setminus \{1_G\}}=0^{G\setminus \{1_G\}}$. As $Z_N=\sigma_s(V_N)_{\{1_G\}}= V$, there exists a configuration 
$d \in V_N$ such that 
\[
\sigma_s(d)(1_G)=c(1_G)=y(1_G)-\sigma_s(x)(1_G).
\]
\par 
In particular, $d\vert_{G\setminus E_N}=0^{G\setminus E_N}$ and $\sigma_s(d)\vert_{G\setminus \{1_G\}}= 0^{G\setminus \{1_G\}}$.  
Let us define $x'= d+x \in V^G$. Then for $g \in G$, the linearity of $\sigma_s$ implies that:  
\begin{align*} 
\sigma_s(x')(g) & =  \sigma_s (d)(g) + \sigma_s(x)(g) \\
&= 
  \begin{cases}
		\sigma_s(x)(g)   & \mbox{if } g \in G\setminus \{1_G\} \\
		y(1_G) - \sigma_s(x)(1_G) + \sigma_s(x)(1_G)  & \mbox{if } g=1_G
	\end{cases}
\\ 
&= 
  \begin{cases}
		 y(g)  & \mbox{if } g \in G\setminus \{1_G\} \\
		y(1_G) & \mbox{if } g=1_G
	\end{cases}
\\ 
& = y(g). 
\end{align*}
\par 
It follows that $\sigma_s(x')=y$. Moreover,  as   $d\vert_{G\setminus E_N}=0^{G\setminus E_N}$ and $x'=d+x$, we have  $x'\vert_{G \setminus E_N}= x\vert_{G \setminus E_N}$. The proof is complete. 
\end{proof}

\section{Stably post-surjective NUCA} 
\label{s:stably-post0surjective}

As for stable injectivity, we define explicitly the stable post-surjectivity property for linear NUCA as follows.   

\begin{definition}
Let $M$ be a subset of a group $G$ and let $V$ be a vector space. 
Let $s \in S^G$ where $S = \LL(V^M,V)$.  The linear NUCA $\sigma_s$ is said to be  \emph{stably post-surjective} if for every $p \in \Sigma(s)$, the NUCA $\sigma_p$ is post-surjective.
\end{definition}
\par 
With the above notation,  suppose that $\sigma_s$ is stably post-surjective for some $s \in S^G$. Then $\sigma_p$ is also stably post-surjective for every $p \in \Sigma(s)$ since 
$\Sigma(p) \subset \Sigma(s)$ (see the proof of \cite[Lemma~5.1]{phung-tcs}).  
\par 
We have the following  fundamental uniform post-surjectivity property of stably post-surjective linear NUCA over finite alphabets. 

\begin{lemma}
[Uniform post-surjectivity] 
\label{l:uniform-post-surjectivity-all} 
Let $M$  be a finite subset of a  countable group $G$. Let $V$ be a finite dimensional vector space over a finite field and let $S = \LL(V^M, V)$. Suppose that $\sigma_s$ is stably post-surjective for some $s \in S^G$. Then there exists a finite subset $E\subset G$ such that for all $g \in G$ and $x, y\in V^G$ with $y\vert_{G \setminus \{g\}} =\sigma_s(x)\vert_{G \setminus \{g\}}$, there exists $z \in A^G$ such that $\sigma_s(z)=y$ and 
$z\vert_{G \setminus gE}= x\vert_{G \setminus gE}$.
\end{lemma}

\begin{proof}
 The lemma is a special case of \cite[Theorem~9.1]{phung-tcs}.  
\end{proof}

\par 
The result \cite[Theorem~13.3]{phung-tcs} specialized to the case of linear NUCA over finite alphabets says that  uniform post-surjectivity is a stable property when passing to the limit of the configurations of local defining maps.

\begin{theorem} 
\label{t:stably-uniform-post-surj-aysn-main}
Let $M$  be a finite subset of a countable group $G$. Let $A$ be a finite vector space and let $s \in S^G$ where  $S=A^{A^M}$. Suppose that $\sigma_s$ is stably  post-surjective. Then there exists $E \subset G$ finite such that for all $p \in \Sigma(s)$, $g \in G$, and $x, y\in A^G$ with  $y\vert_{G \setminus \{g\}} =\sigma_p(x)\vert_{G \setminus \{g\}}$, there exists $z \in A^G$ such that $\sigma_p(z)=y$ and 
$z\vert_{G \setminus gE}= x\vert_{G \setminus gE}$. \qed 
\end{theorem}

\par 
Using Lemma~\ref{l:uniform-post-surjectivity-all} instead of \cite[Lemma~1]{kari-post-surjective} and \cite[Corollary~2]{kari-post-surjective}, the same proof, \textit{mutatis mutandis}, of \cite[Theorem~1]{kari-post-surjective} shows that 
all pre-injective and stably post-surjective linear NUCA with finite memory are invertible. 

\begin{theorem}
\label{t:invertible-general-linear}
Let $G$ be a countable group and let $V$ be a finite vector space. Suppose that $\tau \colon V^G \to V^G$ is a pre-injective and stably post-surjective  linear NUCA with finite memory. Then $\tau$ is invertible.  
\end{theorem}

\begin{proof}
Since $\tau$ is a linear NUCA with finite memory, we can find a finite subset $E \subset G$ and a configuration $s \in S^G$ where $S= \LL(V^E, V)$ such that $\tau= \sigma_s$. Without loss of generality, we can suppose that $E=E^{-1}$, $1_G \in E$, and that $E$  satisfies the conclusion of Theorem~\ref{t:stably-uniform-post-surj-aysn-main}. 
\par 
We define a configuration of local defining maps $t \in S^G$ as follows. Let $g \in G$ and $v \in V^E$. Consider the configuration $x_{g,v} \in V^G$ where $x_{g,v}(h)=v(g^{-1}h)$ if $h \in gE$ and $x_{g,v}(h)=0$ if $h \in G \setminus gE$. Since $x_{g,v}$ is asymptotic to $0^G$ and $\sigma_s(0^G)=0^G$, it follows from the choice of $E$ that there exists $y_{g,v} \in V^G$ asymptotic to $0^G$ such that $\sigma_s(y_{g,v})=x_{g,v}$. Since $\sigma_s$ is pre-injective, such $y_{g,v}$ is unique and we can thus define $t(g)(v)=y_{g,v}(g)$. Note that since $\sigma_s$ is linear, it is clear that $t(g) \colon V^E \to V$ is also a linear map. 
\par
Moreover, let $g \in G$,  $v \in V^E$, and consider any $x \in V^G$ asymptotic to $0^G$ such that $x(h)=v(g^{-1}h)$ if $h \in gE$. Let $y \in V^G$ be the unique configuration asymptotic to $0^G$ such that $\sigma_s(y)=x$. We claim that $y(g)=t(g)(v)$. Indeed, as $x_{g,v}$ and $x$ are asymptotic and agree on $gE$, the set $K \subset G$ on which they disagree is finite and contained in $G \setminus gE$. Consequently, Theorem~\ref{t:stably-uniform-post-surj-aysn-main} and the choice of $E$ imply that $y_{g,v}$ and $y$ can differ at most on $KE \subset (G \setminus gE)E$. Since $E=E^{-1}$, we find that $g \notin (G \setminus gE)E$. It follows that  $y(g)=y_{g,v}(g)=t(g)(v)$ as claimed. 
\par 
Therefore, for every configuration $x \in V^G$ asymptotic to $0^G$, a direct  application of the above property  implies that $\sigma_s(\sigma_t(x))=x$. Indeed, let $y \in V^G$ be the unique configuration asymptotic to $0^G$ such that $\sigma_s(y)=x$ and 
let $F \subset G$ be a finite subset such that $x_{G \setminus F}=0^{G \setminus F}$. Then for every $g \in FE$ and $v \in V^E$ such that $x_{g,v}\vert_{gE}=x\vert_{gE}$, we deduce from the above paragraph that
\[
\sigma_t(x)(g) = t(g)(v) = y(g).
\]
\par 
For $g \in G \setminus FE$, it is clear from the choice of $F$ and $E$ that $\sigma_t(x)(g)=0=y(g)$. Hence, $\sigma_t(x)=y$ and we conclude that $\sigma_s(\sigma_t(x))=\sigma_s(y)=x$. Since $\sigma_s$ and $\sigma_t$ have finite memory and since $x \in V^G$ can be arbitrarily asymptotic to $0^G$, we deduce that $\sigma_s\circ \sigma_t = \Id$.
\par 
Now let $y \in V^G$ be a configuration asymptotic to $0^G$ and let $x=\sigma_s(y)$. Then we have seen that $\sigma_s(y)= \sigma_s(\sigma_t(x))$. As $\sigma_t(x)$ is also asymptotic to $0^G$, we infer from the pre-injectivity of $\sigma_s$ that $y= \sigma_t(x) = \sigma_t(\sigma_s(y))$. As $\sigma_s$ and $\sigma_t$ have finite memory and as $y \in V^G$ can be arbitrarily asymptotic to $0^G$, we can also conclude that $\sigma_t\circ \sigma_s = \Id$. The proof is thus complete. 
\end{proof}

\par 
As an application, we obtain the following characterization of invertible linear NUCA over finite alphabets in terms of pre-injectivity and stable post-surjectivity. 

\begin{theorem}
\label{t:characterization-invertible-linear-NUCA}
Let $G$ be a countable group and let $V$ be a finite vector space. Suppose that $\tau \colon V^G \to V^G$ is a linear NUCA with finite memory. Then $\tau$ is invertible if and only if it is pre-injective and stably post-surjective. 
\end{theorem}

\begin{proof}
   Suppose first that $\tau$ is pre-injective and stably post-surjective. Then Theorem~\ref{t:invertible-general-linear} implies that $\tau$ is invertible. Conversely, assume that $\tau$ is invertible. Then  there exists a linear NUCA $\sigma\colon V^G \to V^G$ with finite memory such that $\tau \circ \sigma= \sigma \circ \tau = \Id$. We infer from Lemma~\ref{l:direct-stably-post-surjective-nuca-linear-finite-a} that $\tau$ is stably post-surjective. Since $\tau$ is invertible, it is also injective and in particular pre-injective. The conclusion thus follows. 
\end{proof}
\par 
Consequently, 
for linear NUCA, it is the stable post-surjectivity, which is a stronger property than post-surjectivity, that correctly makes up for the pre-injectivity to ensure the invertibility. 
\par

\section{Some dual properties of linear NUCA}
\label{s:dual-preoperties-main-result}
We can now state and prove the following dual properties between linear NUCA and their dual linear NUCA. In  particular, the result generalizes the case of CA \cite[Theorem~1.4]{bartholdi-duality-2017}. 

\begin{theorem}
\label{t:dual-properties-linear-dynamic}
Let $M$ be a finite subset of a countable group $G$. Let $V$ be a finite dimensional vector space over a field and let $S = \LL(V^M, V)$.  Then for every  $s \in S^G$, the following hold: 
\begin{enumerate}[\rm (i)]
    \item $\sigma_s$ is pre-injective $\Longleftrightarrow$ $\sigma_{s^*}$ is surjective;  
    \item 
    $\sigma_s$ is injective  $\Longleftrightarrow$   $\sigma_{s^*}$ is post-surjective. 
    \item 
     $\sigma_s$ is stably injective  $\Longleftrightarrow$   $\sigma_{s^*}$ is stably post-surjective. 
\item 
$\sigma_s$ is invertible   $\Longleftrightarrow$   $\sigma_{s^*}$ is invertible. 
\end{enumerate}
\end{theorem}

\begin{proof}
Observe first that $\sigma_s$ is pre-injective if and only if $\Ker(\sigma_s\vert VG)=0$ and $\sigma_{s^*}$ is surjective if and only if $\im\left(\sigma_{s^*}\vert ( V^*)^G\right)=(V^*)^G$. Combining these characterizations with Lemma~\ref{l:dual-property-1}.(i),(iv), the identities $s^{**}=s$, $V^{**}= V$, and the non degeneration of the natural pairing $\langle \cdot \vert \cdot \rangle$,  we deduce immediately that $\sigma_s$ is pre-injective if and only if  $\sigma_{s^*}$ is surjective. Thus,  point (i) is proved. 
\par 
Similarly, for (ii),  it suffices to apply 
Lemma~\ref{l:dual-property-1}.(ii),(iii) and note that $\sigma_s$ is injective if and only if $\Ker(\sigma_s\vert V^G)=0$ and $\sigma_{s^*}$ is post-surjective if and only if $\im\left(\sigma_{s^*}\vert  V^*G\right)=V^*G$. 
\par 
For (iii), suppose first that $\sigma_s$ is stably injective and let $p \in \Sigma(s^*)$. Then we infer from Lemma~\ref{l:duality-family} that $p^* \in \Sigma(s)$. It follows that $\sigma_{p^*}$ is injective. Consequently, point (ii) implies that $\sigma_{p^{**}} = \sigma_p$ is post-surjective. Since $p \in \Sigma(s^*)$ can be arbitrary, we deduce that $\sigma_{s^*}$ is stably post-surjective. By a similar argument, we also find that if $\sigma_{s^*}$ is stably post-surjective then $\sigma_s$ is stably injective. Point (iii) is thus proved. 
\par 
For (iv), suppose that $\sigma_s$ is invertible. Then up to enlarging $M$, there exists $t \in S^G$ such that $\sigma_t\circ \sigma_s= \sigma_s \circ \sigma_t= \Id$. Hence, we infer from Lemma~\ref{l:duality-functoriality} that 
\[
\sigma_{s^*}\circ \sigma_{t^*} = (\sigma_t\circ \sigma_s)^*=\Id^*=\Id, 
\]
and similarly, 
\[
\sigma_{t^*}\circ \sigma_{s^*} = (\sigma_s\circ \sigma_t)^*=\Id^*=\Id. 
\]
\par 
It follows that the dual linear NUCA  $\sigma_{s^*}$ is also invertible. Point (iv) is proved and the proof is thus complete. 
\end{proof}

\begin{proof}[Another proof of Theorem~\ref{t:dual-properties-linear-dynamic}.(iv) when $V$ is finite] 
Suppose that $V$ is finite and $\sigma_s$ is invertible. Then up to enlarging $M$, there exists $t \in S^G$ such that $\sigma_t\circ \sigma_s= \sigma_s \circ \sigma_t= \Id$. We infer from  \cite[Theorem~11.1]{phung-tcs} and the relation $\sigma_t\circ \sigma_s$  that $\sigma_s$ is stably injective. Hence, Theorem~\ref{t:dual-properties-linear-dynamic}.(ii) tells us that $\sigma_{s^*}$ is stably post-surjective. 
On the other hand, the surjectivity of $\sigma_s$, which results from the relation $\sigma_s \circ \sigma_t= \Id$, implies that $\sigma_{s^*}$ is pre-injective by Theorem~\ref{t:dual-properties-linear-dynamic}.(i). 
 Consequently, we can apply Theorem~\ref{t:invertible-general-linear} to conclude that $\sigma_{s^*}$ is invertible. By replacing $s$ by $s^*$ and note that $s^{**}=s$, we find that if $\sigma_{s^*}$ is invertible then so is $\sigma_s$. Theorem~\ref{t:dual-properties-linear-dynamic}.(iv) is thus proved. 
\end{proof}

\section{Characterizations of stable injectivity and stable post-surjectivity for linear NUCA} 
\label{s:characterization-left-inevertible}

In this section, we will characterize the stable injectivity property and the stable post-surjectivity property respectively in terms of left invertibility and right invertibility 
in the case of linear NUCA. 

\par 
\begin{theorem}
    \label{t:characterization-left-right-invertibility}
Let $V$ be a finite  vector space and let $G$ be a countable group. Suppose that $\sigma \colon V^G \to V^G$ is a linear NUCA with finite memory. Then the following hold: 
\begin{enumerate}[\rm (i)]
    \item 
    $\sigma$ is stably injective if and only if  $\tau \circ \sigma= \Id$ for some linear NUCA $\tau \colon V^G \to V^G$ with finite memory;
    \item 
     $\sigma$ is stably post-surjective  if and only if  $\sigma \circ \tau = \Id$ for some linear NUCA $\tau \colon V^G \to V^G$ with finite memory. 
\end{enumerate}
\end{theorem}

\begin{proof}
    By \cite[Theorem~A]{phung-tcs}, we know that $\sigma$ is stably injective if and only if it is left invertible, i.e., there exists a NUCA $\tau \colon V^G \to V^G$ with finite memory such that $\tau \circ \sigma = \Id$. Since $\sigma$ is linear, it is immediate that $\tau$ is also a linear NUCA. Assertion (i) is thus proved. 
    \par 
    For (ii), let $\sigma^*$ be the dual linear NUCA of $\sigma$. It follows from Theorem~\ref{t:dual-properties-linear-dynamic}.(ii) that $\sigma=(\sigma^*)^*$ is stably post-surjective if and only if $\sigma^*$ is stably injective. On the other hand, assertion (i) implies that $\sigma^*$ is stably injective if and only if there exists a linear NUCA $\tau\colon V^{*G} \to V^{*G}$ with finite memory such that $\tau \circ \sigma^*= \Id$, or equivalently, $\sigma \circ \tau^*=\Id$ by 
    Lemma~\ref{l:duality-functoriality}.  Since $\tau^* \colon V^G \to V^G$ is a linear NUCA with finite memory, we conclude that $\sigma$ is stably post-surjective if and only if $\sigma \circ \pi=\Id$ for some linear NUCA $\tau\colon V^{G} \to V^{G}$ with finite memory. Assertion (ii) follows and the proof is complete. 
\end{proof}

\section{Applications on the dual surjunctivity} 
\label{s:dual-surjunctivity}

As a consequence of Theorem~\ref{t:dual-properties-linear-dynamic}, we obtain the following dual result of \cite[Theorem~B]{phung-tcs} for linear NUCA.

\begin{corollary}
Let $M$ be a finite subset of a countable group $G$. Let $V$ be a finite  vector space and let $S = \LL(V^M, V)$. Let $s \in S^G$ be asymptotic to a constant configuration.  
Then $\sigma_s$ is invertible in each of the following cases: 
\begin{enumerate} [\rm (i)]
\item $G$ is amenable and $\sigma_s$ is post-surjective;
\item $G$ is residually finite and $\sigma_s$ is stably post-surjective. 
\end{enumerate}
\end{corollary}

\begin{proof}
By Theorem~\ref{t:dual-properties-linear-dynamic}.(ii)-(iii), we know that $\sigma_s$ is post-surjective, resp. stably post-surjective, if and only if $\sigma_{s^*}$ is injective, resp. stably injective. Moreover, Theorem~\ref{t:dual-properties-linear-dynamic}.(iv) implies that  $\sigma_s$ is invertible  if and only if so is $\sigma_{s^*}$. We deduce that the corollary is in fact equivalent to \cite[Theorem~B]{phung-tcs} specialized to the case of linear NUCA over finite alphabets. The proof is thus complete.  
\end{proof}

\par 
Similarly, the dual linear version of \cite[Theorem~C]{phung-tcs} can be stated as follows: 

\begin{corollary}
\label{t:dual-linear-singularity}
Let $M \subset \Z^d$ be a finite subset. Let $V$ be a finite vector space   and let $S=\LL(V^M,V)$. Suppose that  for some $s \in S^{\Z^d}$, the linear  NUCA  $\sigma_s$ is stably post-surjective with bounded singularity. 
Then $\sigma_s$ is invertible. \qed
\end{corollary} 

Here, we say that     $\sigma_s$ has \emph{bounded singularity} if for all finite subset $E \subset \Z^d$, there exists a box $K = \prod_{j=1}^d\llbracket a_j,b_j \rrbracket^d \subset \Z^d$ where $\llbracket a_j,b_j \rrbracket = \{a_j,  \dots, b_j\}$ which contains $E$ such that  
$s(g) = s(k)$ for all $g \in KE \setminus K$ and the unique $k \in K$ with  $k \equiv g$ (mod $K$), i.e.,  
$k_j \equiv g_j$ (mod $a_j-b_j+1$) for every $j=1, \dots, d$ where $g=(g_1, \dots, g_d)$ and $k = (k_1, \dots, k_d)$.

\section{Bijectivity is not a dual property for linear NUCA}
\label{s:bijectivity-not-dual}

In this section, we show that in contrast to the case of linear CA over finite dimensional vector space alphabets, the equivalence between the bijectivity of a linear NUCA and the bijectivity of its dual linear NUCA is  no longer valid. Consequently, we infer from Theorem~\ref{t:dual-properties-linear-dynamic} that bijectivity is not equivalent to pre-injectivity and post-surjectivity combined. 
In particular, the result \cite[Theorem~1]{kari-post-surjective} for CA fails for the more general class of NUCA. 

\begin{proposition}
\label{p:bijective-non-dual}
Let $A= \Z/2\Z$. There exists a bijective non-post-surjective linear $\mathrm{NUCA}$ $\tau \colon A^\Z \to A^\Z$. Moreover, its dual linear $\mathrm{NUCA}$ $\tau^* \colon V^\Z \to V^\Z$ is pre-injective and post-surjective but is not injective.  
\end{proposition}

Proposition~\ref{p:bijective-non-dual} will follow immediately from the next example.  We consider the following linear NUCA  with finite memory described in  \cite[Example~14.2]{phung-tcs} where it was shown 
to be bijective non-reversible and non-stably injective. 

\begin{example} 
\label{ex:1}
Let $G= \Z$ and let $M=\{-1,0\}$. Consider the linear maps $f,g \colon A^M \to A$ defined for all $(u,v)\in A^M$ by: 
\begin{equation}
    \label{e:counter-1-2-3}
f(u,v) = v,  \quad g(u,v) =  u+v \quad (\text{mod }2). 
\end{equation} 
\par 
Let $S=A^{A^M}$. For $k \in \Z\cup \{\pm \infty\}$, we  define $s_k \in S^\Z$  by setting $s_k(n)=f$ if $n \leq k$ and $s_k(n)=g$ if $n \geq k+1$.  
Let $s=s_0$ then  $\sigma_s\colon A^\Z \to A^\Z$ is a linear NUCA with finite memory which is bijective but not reversible and thus not stably injective (see \cite[Example~14.2]{phung-tcs}).  
\par 
Let $x \in A^G$ and  $y=\sigma_s(x)$. Then \cite[Example~14.2]{phung-tcs} shows that $x(n)=y(n)$ for $n \leq 0$ and if $n \geq 1$, we have:  
\[
x(n)= 
y(n) - y(n-1) + \dots + (-1)^{n} y(0) \text{ (mod }2).  
\]
\par 
Consequently, if $y' \in A^G$ is the configuration defined by $y'(0)=y(0)+1$ and $y'(n)=y(n)$ for all $n \neq 0$, then for the unique $x' \in A^G$ such that $\sigma_s(x')=y'$, we have $x'(n)=x(n)+1$ for all $n \geq 1$. In particular, $x'$ and $x$ are not asymptotic while $y'$ and $y$ are. This shows that $\sigma_s$ is surjective but not post-surjective. 
\par 
Therefore, Theorem~\ref{t:dual-properties-linear-dynamic}.(i) and (ii) respectively imply that the dual linear  NUCA $\sigma_{s^*}$ is pre-injective and post-surjective. However,   we infer from Theorem~\ref{t:dual-properties-linear-dynamic}.(iii) that $\sigma_{s^*}$ is not injective and thus not bijective. 
\par 
To illustrate, we shall reprove the above properties by working  directly with $\sigma_{s^*}$. Let $N=\{0,1\}$ and define the linear maps $f^*, g^*\colon V^N \to V$ by:
\[
f^*(u,v) = u, \quad g^*(u,v) = u+v \quad \text{(mod 2)}, \quad (u,v) \in V^N.
\]
\par 
Then by Definition~\ref{d:dual-config-local-map}, $s^*(n)=f^*$ if $n \leq -1$ and $s^*(n)=g^*$ if $n \geq 0$. Let $c \in V^G\setminus \{0^G\}$ be the configuration defined by $c(n)=0$ if $n \leq -1$ and $c(n)=1$ if $n \geq 0$. It is clear that $\sigma_{s^*}(c)=0^G$ and thus  $\sigma_{s^*}$ is not injective. 
\par 
In fact, a straightforward computation shows that $x \in \Ker(\sigma_{s^*})$ if and only if $x=0^G$ or $x=c$, depending on whether $x(0)=0$ or $x(0)=1$. In particular, $\sigma_{s^*}$ is pre-injective. To check that $\sigma_{s^*}$ is post-surjective, let $x \in V^G$ be asymptotic to $0^G$. Let $n_0 \geq 1$ be large enough such that $x(k)=0$ whenever $\vert k \vert \geq n_0$. Consider $y \in V^G$ where $y(k)=0$ for $\vert k \vert \geq n_0$, $y(k)=x(k)$ if $n_0< k \leq -1$, and for $0 \leq k <n_0$, we define: 
\[
y(k) = x(k) - x(k-1) + \dots +(-1)^{n_0-k+1}x(n_0-1). 
\]
\par 
Then a direct computation implies that $\sigma_{s^*}(y)=x$. Since $y$ is asymptotic to $0^G$, we conclude that $\sigma_{s^*}$ is post-surjective. 
\end{example}

\section{Column factorizations} 
\label{s:kurka-construction} 
We recall the notion of \emph{column factorizations} or \emph{canonical factors} of Kurka \cite{kurka-97} and  Blanchard-Maass \cite{blanchard-maass-97}) which was later generalized in \cite{phung-shadowing}.  
Let $G$ be a countable group and let $A$ be a set. 
Suppose that $\tau_1, \dots, \tau_r\colon A^G \to A^G$  
are pairwise commuting NUCA. We define $\tau_\alpha = \tau_1^{a_1}   \dots    \tau_r^{a_r}$ 
as the composition $\tau_1^{a_1} \circ   \dots \circ    \tau_r^{a_r}$ for  $\alpha= (\alpha_1, \dots, \alpha_r) \in \N^r$. 
Let $E \subset G$ be a finite subset. 
We construct a map $\Psi_E \colon A^G \to  (A^E)^{\N^r} $ by setting:     
\begin{equation} 
\label{e:psi-e}
\Psi_E(x) (\alpha)  \coloneqq  \left( \tau_\alpha (x) \right)\vert_E, \quad \text{ for every }x \in A^G \text{ and } \alpha \in \N^r. 
\end{equation}
\par 
\begin{definition} 
The \emph{column factorization} associated with the data $E$ and $\tau_1, \dots, \tau_r$ is defined by: 
\begin{equation}
\label{e:column-facto}
\Lambda(E ;  \tau_1, \dots, \tau_r) \coloneqq \Psi_E(A^G) \subset (A^E)^{\N^r}. 
\end{equation}
\end{definition}
\par 
For every set $X$, the shift action of the additive monoid $(\N^G, +)$ on the full shift $X^{\N^r}$ is given by $(\alpha x)(\beta) = x(\alpha + \beta) $ for all $x \in X^{\N^r}$ and $\alpha, \beta \in \N^r$. 
\par 
We have the following crucial property of column factorizations whose proof follows closely the proof of \cite[Theorem~7.2]{phung-shadowing}.

\begin{theorem} 
\label{l:closed-kurka-sft}
Let $E$ be a finite subset of a countable group $G$.  Let $A$ be a finite dimensional vector space and   let $\tau_1, \dots, \tau_r \colon A^G \to A^G$ be pairwise commuting linear NUCA with finite memory. 
Then $\Lambda(E ;  \tau_1, \dots, \tau_r) \subset (A^E)^{\N^r}$  is  a subshift of finite type, i.e., there exists a finite subset $
\Omega \subset \N^r$ such that $\Lambda(E ;  \tau_1, \dots, \tau_r) = \{x \in (A^{ E})^{\N^r}\colon (\alpha x)\vert_{\Omega} \in \Lambda_{\Omega} \text{ for all } \alpha \in \N^r\}$.  
\end{theorem}

\begin{proof}
We claim that 
$\Lambda \coloneqq \Lambda(E ;  \tau_1, \dots, \tau_r) = \Psi_E(A^G)$ is a closed linear subshift of $(A^E)^{\N^r}$. 
First, note that $\Lambda$ is a linear subspace of $(A^E)^{\N^r}$ as    $\tau_\alpha$ is linear for all $\alpha \in \N^r$.  
Let $x \in A^G$,  $y= \Psi_E(x) \in \Lambda$, and $\alpha, \beta \in \N^r$. 
As $\tau_1, \dots, \tau_r$ are pairwise commuting, 
one has: 
\begin{align*} 
(\beta  y)( \alpha ) & = y ( \alpha + \beta ) 
= \Psi_E( x ) ( \alpha + \beta )  = \tau_{\alpha + \beta}(x)\vert_E 
=  \tau_\alpha( \tau_\beta  (x) )\vert_E \nonumber  
\\ &= \Psi_E(\tau_\beta(x))(\alpha) \nonumber
\end{align*} 
\par 
It follows that $\beta y = \Psi_E(\tau_\beta(x)) \in \Lambda$ and thus 
$\Lambda$ is a subshift of $(A^E)^{\N^r}$. 
\par 
For every finite subset $F \subset \N^r$, we choose a large enough finite subset $M_F \subset G$
such that $M_F$ is memory set of  $\tau_\alpha$ for all  
$\alpha \in F$. 
Consider the linear map $\tau_F \coloneqq \left(  \tau_{\alpha}\right)_{\alpha \in F} \colon A^G \to (A^F)^G$ defined by: 
\begin{equation} 
\label{e:tau-F-monoid-main}
\tau_F (x) (g) \coloneqq (\tau_\alpha(x)(g))_{\alpha \in F}, \quad \mbox{for all $x \in A^G$ and $g \in G$}. 
\end{equation} 
\par 
Let $z \in (A^E)^{\N^r}$ that belongs to the closure of $\Lambda$ 
in $(A^E)^{\N^r}$. 
 Let $(F_n)_{n \geq 1}$ be  an exhaustion of $\N^r$ consisting of finite subsets. Since $G$ is countable, we can suppose that  $(M_{F_n})_{n \geq 1}$ also forms 
 an exhaustion of $G$.    
 \par    
 For every $n \geq 1$, we define  
$\tau_{n, E}^{+} \colon A^{EM_{F_n}} \to (A^{F_n})^{E}$ by 
$\tau_{n, E}^{+}(c) \coloneqq \tau_{F_n}(x)\vert_E$ for all $c \in A^{EM_{F_n}}$ and $x \in A^G$ such that $x \vert_{EM_{F_n}}=c$. 
\par 
Let $(E_n)_{n \geq 1}$ be an exhaustion by finite subsets of $G$ such that 
$EM_{F_n} \subset E_n$ for all $n \geq 1$. Thus, we obtain the linear  projections 
 $p_n \colon A^{E_n} \to A^{M_{F_n}E}$. 
 Let us denote $f_n \coloneqq \tau_{n, E}^{+} \circ p_n\colon A^{E_n}  \to (A^{F_n})^{E}$ for every $n\geq 1$ and 
consider: 
 \begin{align} 
 \label{e:x-n-z} 
 X_{n} (z) \coloneqq \{ x \in A^{E_n} \colon  
  f_n (x) = z\vert_{F_n}   \} = f_n^{-1} (z\vert_{F_n}) \subset A^{E_n}. 
 \end{align}
 \par 
 As $z$ belongs to the closure of $\Lambda$ in   $(A^E)^{\N^r}$, 
 it is immediate from the definition of $\Lambda$ and \eqref{e:tau-F-monoid-main}, \eqref{e:x-n-z} 
 that $X_{n}(z)$ is nonempty for all $n \geq 1$.
 \par 
 Since  $p_n$ and  $\tau^+_{n,E}$ are  linear, so is  $f_n=\tau_{n, E}^{+} \circ p_n$. It follows from \eqref{e:x-n-z} that   
 $X_n(z)$ is an affine subspace of $A^{E_n}$ and $A^{E_n}$ 
 for all $n \geq 1$. 
We thus obtain an inverse system $(X_n(z))_{n \geq 1}$ 
whose transition maps $X_m(z) \to X_n(z)$ ($m \geq n \geq 1$)  
are the restrictions of the linear projections $\pi_{m,n} \colon A^{E_m} \to A^{E_n}$. 
We infer from \cite[Lemma~3.1]{cscp-jpaa} that  
there exists 
\[
x \in \varprojlim_{n \geq 1} X_{n}(z) \subset \varprojlim_{n\geq 1} A^{E_n} = A^G. 
\] 
\par 
By construction,  $ z= \Psi_{E}(x)\in \Lambda$. Thus,  $\Lambda$ is closed in $(A^E)^{\N^r}$ with respect to 
the prodiscrete topology. 
We conclude that $\Lambda$ is a closed linear subshift of $(A^E)^{\N^r}$. Hence,  $\Lambda$   is a subshift of finite type by \cite[Theorem~6.5]{phung-shadowing}.    
\end{proof}

\section{Shadowing property for linear NUCA} 
\label{s:main-result}

\begin{definition}
\label{d:pseudo-orbit} 
Let $X$ be a set. 
Let $S$ be a finitely generating set of a monoid $\Gamma$. 
Let $T$ be an action of $\Gamma$ on $X$  
and let $d$ be a metric on $X$. 
\begin{enumerate} [\rm (i)]
\item 
for $\delta >0$, a sequence  $(x_\tau)_{\tau\in \Gamma}$ in $X$ is called an \textit{$(S,d,\delta)$-pseudo-orbit} of $T$  if $d(T(\sigma,x_\tau), x_{\sigma \tau}) < \delta$ for all $\sigma \in S$ and $\tau\in \Gamma$.
 \item 
the action $T$ has the \textit{$(S,d)$-shadowing property}   
if for every $\varepsilon>0$, there exists $\delta>0$ such that every $(S,d,\delta)$-pseudo-orbit  $\{x_\tau\}_{\tau\in \Gamma}$ 
of $T$ is \emph{$\varepsilon$-shadowed} by some point $x$ of $X$, i.e., 
$d(T(\tau,x), x_\tau) < \varepsilon$ for all $\tau\in \Gamma$.
\end{enumerate} 
\end{definition}

We have the following intrinsic notion of shadowing property for actions of a finitely generated monoid on the full shift (cf. \cite[Definition~3.3]{phung-shadowing}). 

\begin{definition}
\label{d:shadow-general} 
Let $A$ be a set and let $G$ be a countable group. 
Let $\Gamma$ be a finitely generated monoid. 
An action $T$ of $\Gamma$ on $A^G$  has the \textit{shadowing property} 
if $T$ has the $(S,d)$-shadowing property for 
 every finitely generating set $S$ of $\Gamma$ 
and every standard metric $d$ on $A^G$. 
\end{definition}
\par 
Here,  a metric $d$ on $A^G$ is \emph{standard} if
there exists an exhaustion $(E_n)_{n \geq 0}$  of $G$ by finite subsets such that 
for every $x, y \in A^G$, 
\begin{align}
\label{e:metric-main}
d(x,y) \coloneqq 2^{-n}, \mbox{ where }  n \coloneqq \sup \{ k \in \N \colon x\vert_{E_k} =y\vert_{E_k} \}. 
\end{align}
\par

We establish the following  result  whose proof is,  \emph{mutatis mutandis}, the same as the proof of \cite[Theorem~8.1]{phung-shadowing}. Note that linear NUCA with finite memory are Lipschitz. Hence, \cite[Lemma~4.2]{phung-shadowing}, which is essential in the proof of \cite[Theorem~8.1]{phung-shadowing}, also applies  for linear NUCA. 

\begin{theorem} 
\label{t:main-shadow}
Let $G$ be a countable group. Let $A$ be a finite dimensional vector space.  
Let $\Gamma$ be a finitely generated abelian monoid of linear NUCA $A^G \to A^G$ with finite memory. Then the natural action of $\Gamma$ on $A^G$ has the shadowing property.   
\end{theorem}

\begin{proof} 
Suppose that $S= \{ \tau_1, \dots, \tau_r \}$ is a finite generating set 
of $\Gamma$.  
For every $n \geq 1$, let $F_n \coloneqq  \{(a_1, \dots, a_r) \in \N^r \colon a_1, \dots,a_r \leq n \}$ then $F_n \subset F_{n+1}$ and $\N^r  = \bigcup_{n \geq 1} F_n$. 
As $G$ is countable, it admits an exhaustion $(E_n)_{n \geq 0}$ of $G$ by  finite subsets such that   
$E_0= \varnothing$, $1_G \in E_1$, 
$E_n^2  \subset E_{n+1}$ for every $n \geq 0$.  
Thus, 
$E_n \subset E_{n+1}$ for every $n \geq 0$. Moreover, 
for every $M \subset G$ finite, there exists $n_0 \geq 1$ such that 
$E_n M \subset E_{n+1}$ for all $n \geq n_0$ (cf.~\cite[Definition~4.1]{phung-shadowing}).  
Let  $d$ be the standard metric on $A^G$ associated with $(E_n)_{n \geq 0}$.  
Let us fix $\varepsilon >0$ and choose  $n_0 >0$ such that 
$2^{-n_0} < \varepsilon$.   
\par 
The column factorization (cf.~Section~\ref{s:kurka-construction}) 
$\Lambda \coloneqq \Lambda(E ;  \tau_1, \dots, \tau_r)$ is 
a linear subshift of finite type of $ (A^{E_{n_0}})^{\N^r}$
by Theorem~\ref{l:closed-kurka-sft}. 
Thus, there exists  $N \geq 1$ such that 
$\Lambda = \{x \in (A^{E_{n_0}})^{\N^r}\colon (\alpha x)\vert_{F_N} \in \Lambda_{F_N} \text{ for all } \alpha \in \N^r\}$.   
\par 
Since $F_N$ is finite, \cite[Lemma~4.2]{phung-shadowing}  implies that there exists a finite constant $C \geq 1$ such that every 
$\tau_\alpha $, $\alpha  \in F_N$, 
 is $C$-Lipschitz, i.e.,
\begin{equation} 
\label{e:c-lipschitz-commun}
d(\tau_\alpha(x), \tau_\alpha(y)) \leq C d(x,y), \quad \text{ for every }x, y \in A^G.   
\end{equation}
\par 
Denote $\delta  \coloneqq  \frac{1}{2^{n_0}CNr}$.  
Let $(x_\tau)_{\tau \in \Gamma}$ be an $(S, d ,\delta)$-pseudo-orbit of the natural action $T$ of $\Gamma$ on $A^G$ where $T(\tau, x) \coloneqq \tau(x)$ for every $\tau \in \Gamma$ and $x \in A^G$.   
Then by definition, we have  for all $\sigma \in S$ and $\tau \in \Gamma$ that 
\begin{equation} 
\label{e:delta-1}
d(T(\sigma ,x_\tau), x_{\sigma \tau}) < \delta.  
\end{equation}
\par 
Let $\alpha= (a_1, \dots, a_r)  \in F_N$ and $\tau \in \Gamma$. 
We infer from the triangle inequality, the $C$-Lipschitz continuity of $\tau_{\beta}$ for every $\beta \in F_N$ (cf.~\eqref{e:c-lipschitz-commun}),  
and   the choice of $\delta$ that (see the proof of \cite[Theorem~8.1]{phung-shadowing}): 
 \begin{align} 
\label{e:delta-epsilon-shadowing} 
d ( T ( \tau_{\alpha} ,x_\tau), x_{\tau_\alpha \tau} )  
& \leq C \sum_{k=0}^{a_1-1}  d (   \tau_1(x_{\tau_1^{a_1-k-1}   \dots    \tau_r^{a_r} \tau }), x_{\tau_1^{a_1- k}   \dots    \tau_r^{a_r} \tau })  \,\,  + 
  \\
&  + C \sum_{k=0}^{a_2-1}  d ( \tau_2(x_{\tau_1^{a_2-k-1}   \dots    \tau_r^{a_r} \tau }), x_{\tau_2^{a_2- k}   \dots    \tau_r^{a_r} \tau })   \,\, +   \nonumber \\
&   \dots  && \nonumber \\
&   + C \sum_{k=0}^{a_r-1}  d (\tau_r(x_{\tau_r^{a_r-k-1} \tau }),  x_{\tau_r^{a_r-k} \tau }) \nonumber  \\ 
& \leq C (a_1 + \dots + a_r) \delta  \nonumber\\
& \leq CNr \frac{1}{2^{n_0}CNr} = 2^{-n_0}.  \nonumber
\end{align}  
\par 
Consider  $z \in (A^{E_{n_0}})^{\N^r}$ where 
$z(\alpha) \coloneqq x_{\tau_\alpha}\vert_{E_{n_0}}$ 
for every $\alpha \in \N^r$. 
We claim that $z \in \Lambda$. Indeed, let $\beta \in \N^r$. Then by \eqref{e:delta-epsilon-shadowing} and by the choice of $n_0$, 
we have for every $\alpha \in F_N$ that 
\begin{align} 
  (\tau_{\alpha} (x_{\tau_\beta}))\vert_{E_{n_0}} =  x_{\tau_\alpha \tau_\beta} \vert_{E_{n_0}}=x_{\tau_{\alpha + \beta}}\vert_{E_{n_0}}. 
\end{align} 
\par 
Thus, \eqref{e:psi-e} implies that 
$z\vert_{\beta + F_N} \in \Lambda_{\beta + F_N}$. Hence $z \in \Lambda$ and  
there exists $x \in A^G$ 
such that $ \Psi_{E_{n_0}} (x) = z$. 
We deduce from the definitions of $\Psi_{E_{n_0}}$ and $z$ 
that for all $\alpha \in \N^r$, we have 
$(\tau_\alpha(x))\vert_{E_{n_0}}= z(\alpha) = x_{\tau_\alpha}\vert_{E_{n_0}}$.  
\par 
Let $\tau\in \Gamma$, then $\tau = \tau_\alpha$ 
for some $\alpha \in \N^r$ and: 
\[
d(T(\tau,x), x_\tau) = d(\tau_\alpha(x), x_{\tau_\alpha}) \leq 2^{-n_0} < \varepsilon. 
\]
\par 
Thus,   
$x$  $\varepsilon$-shadows the $(S, d,\delta)$-pseudo-orbit $(x_\tau)_{\tau \in \Gamma}$ 
with respect to the standard metric $d$ (cf.~\eqref{e:metric-main}).  
The proof is complete.  
\end{proof}

We thus obtain the following direct consequence. 
 
\begin{corollary}
\label{c:shift}
Let $G$ be a finitely generated abelian group. Let $V$ be a finite dimensional vector space. 
Then the shift action of $G$ on $A^G$ has the shadowing property. 
\end{corollary}

\bibliographystyle{siam}

\end{document}